\documentclass[12pt,a4]{amsart}
\usepackage[T1]{fontenc}
\usepackage{latexsym} 
\usepackage{amsfonts, amsthm, amsmath,amssymb}
\usepackage[latin1]{inputenc}
\usepackage{times}
\usepackage{color}
\usepackage{xcolor}
\usepackage{mdframed}
\usepackage{framed}
\definecolor{shadecolor}{rgb}{1,1,1}
\definecolor{refkey}{rgb}{0,0,1}
\definecolor{labelkey}{rgb}{1,0,0}

\def\<{{\langle}} 
\def\>{{\rangle}}

\def\note#1{{}}

\def\note#1{} 

\def\beq{\begin{equation}} 
\def\eeq{\end{equation}}

\def\id{\mathrm{id}}

\newcounter{zlist} 

\newcounter{blist} 
 
\newcounter{rlist} 
 
 \def\stac#1{\raise-.2cm\hbox{$\stackrel{\displaystyle\otimes}{\scriptscriptstyle{#1}}$}}

\def\cten#1{\raise-.2cm\hbox{$\stackrel{\displaystyle\widehat{\otimes}}{\scriptscriptstyle{#1}}$}}

\def\Label#1{\label{#1}\ifmmode\llap{[#1] }\else 
\marginpar{\smash{\hbox{\tiny [#1]}}}\fi} 
\def\Label{\label} 
\newtheorem{proposition}{Proposition}[section]

\theoremstyle{definition} 
\newtheorem{definition}[proposition]{Definition}

\theoremstyle{remark} 
\newtheorem{remark}[proposition]{Remark} 
 
\newcounter{c} 
 
\newcommand{\etyk}[1]{\vspace{-7.4mm}$$\begin{equation}\Label{#1} 
\addtocounter{c}{1}} 
\renewcommand{\]}{\ifnum \value{c}=1 $$\else \end{equation}\fi} 
\def\TT{{\mathbb T}}

\newcommand{\Cc}{\mathcal{C}}

\def\*C{{}^*\hspace*{-1pt}{\Cc}}
\def\text#1{{\rm {\rm #1}}}

\def\1{\mathbf{1}}

\newcounter{mnotecount}[section]
\renewcommand{\themnotecount}{\thesection.\arabic{mnotecount}}
\newcommand{\mnote}[1]
{\protect{\stepcounter{mnotecount}}$^{\mbox{\footnotesize
$
\bullet$\themnotecount}}$ \marginpar{
\raggedright \tiny\em
$\bullet$\themnotecount: #1} }

\begin{document} 
\vspace*{-2cm}
\title{Spectral triples with multitwisted real structure.}

\author[L.\ D\k{a}browski]{Ludwik D\k{a}browski${}^\dagger$}
\address{SISSA (Scuola Internazionale Superiore di Studi Avanzati), \newline\indent Via Bonomea 265, 34136 Trieste, Italy} 
\email{dabrow@sissa.it} 

\author[A.\ Sitarz]{Andrzej Sitarz${}^\ddagger$}
\thanks{${}^\dagger$ \scriptsize Partially supported by H2020-MSCA-RISE-2015-691246-QUANTUM DYNAMICS}
\thanks{${}^\ddagger$ \scriptsize  Partially supported by the Polish National Science Centre grant 2016/21/B/ST1/02438}
\address{Institute of Theoretical Physics, Jagiellonian University, \newline\indent
	prof.\ Stanis\l awa \L ojasiewicza 11, 30-348 Krak\'ow, Poland.}
\email{andrzej.sitarz@uj.edu.pl}   

\begin{abstract} 
We generalize the notion of spectral triple with reality structure to spectral triples with multitwisted real structure,  
the class of which is closed under the tensor product composition. In particular, we introduce a multitwisted 
order one condition  (characterizing the Dirac operators as an analogue of first-order differential operator).
This provides a unified description of the known examples, which include rescaled triples with the conformal factor from the commutant 
of the algebra and (on the algebraic level) triples on quantum disc and on quantum cone, that satisfy  twisted first order 
condition of \cite{BCDS16,BDS19}, as well as asymmetric tori, non-scalar conformal rescaling and noncommutative 
circle bundles. In order to deal with them we allow twists that do not implement automorphisms of the algebra of 
spectral triple.
\vspace{3mm}
\end{abstract} 
\maketitle 
 \vspace*{-1cm}
\section{Introduction}

Spectral triples were introduced  \cite{CoBk94} as a setup to generalize differential geometry to noncommutative algebras 
that carries topological information and allows explicit analytic computations of index parings \cite{CoMo95}.
The concept of {\em real spectral triples} \cite{Co95} was motivated by successful applications to the Standard Model of particle 
physics and also by the quest for the equivalence in the commutative case with the geometry of spin manifolds, culminating in 
the reconstruction theorem \cite{Co13}. The role of the real structure in noncommutative examples became evident
in the relation between the classes of equivariant real spectral triples and the spin structures on noncommutative tori \cite{PaSi06}.

While the theory of real spectral triples gained more and more examples \cite{CoLa01, GGBISV, DLSSV}, some 
interesting noncommutative geometries did not fit into the original set of axioms for real spectral triples
(unbounded K-cycles). 
Remarkable ones were the twisted (or modular) spectral triples on the curved noncommutative torus 
\cite{CoTr11}, intensively studied afterwards (see \cite{FK} for a review of results of curvature computations). 
A scheme to incorporate the noncommutative analogue of conformally 
rescaled geometries in the framework of usual spectral triples was proposed in \cite{BCDS16}. 
Therein the Dirac operator is rescaled by a positive element from the commutant of the algebra, thus 
mantaining the bounded commutator with elements of the algebra, but leading to a twisted reality structure, together with a generalized first order condition. This construction was further studied in \cite{BDS19}, where the relation between spectral triples 
with twisted real structure and real twisted spectral triples \cite{LaMa} was uncovered. 

Yet even these generalizations do not embrace the recent examples of partially rescaled conformal torus \cite{DaSi15} 
and spectral triples over a circle bundle with the Dirac operator compatible with a given connection \cite{DaSi13}. 
Moreover, neither the class of spectral triples with a twisted first order condition nor the twisted spectral triples 
is closed under the tensor product composition of spectral triples. We propose here a construction 
which complies with tensor product of spectral triples, allows for fluctuations and covers almost all known interesting 
and geometrically motivated examples. To avoid confusion with {\em twisted spectral triples} (and {\em real twisted spectral triples}) 
we use the name {\em spectral triples with (multi)twisted real structure} or for brevity {\em (multi)twisted-real spectral triples}. 

\section{Multitwisted real structure for spectral triples.}

Consider a spectral triple $(A,H,D)$, where $A$ is a $*$-algebra identified with 
a subalgebra of bounded operators $B(H)$ on a Hilbert space $H$,
and $D$ is a densely defined selfadjoint operator on $H$ 
such that $D$ has a compact resolvent and for each $a\!\in\!A$ the commutator 
$[D,a]$ is bounded. Let $J$ be an antilinear isometry  on $H$, 
such that $J^2 = \pm 1$ and
\begin{equation} \label{0oc}
[a,J bJ^{-1} ] =  0,
\end{equation}
in which case (with a slight abuse of terminology) we call 
$(A,H,D,J)$ a real spectral triple. If in addition there is a grading 
$\gamma$ of $H$, $\gamma^2=1$, such that $D\gamma=-\gamma D$, 
$[\gamma,a]=0$ for all $a$ in $A$ 
and $\gamma J=\pm J\gamma$,
we call $(A,H,D,J,\gamma)$ a real even spectral triple.

\begin{definition}\label{multi}
We say that a real spectral triple $(A,H,D,J)$ is 
a {\em spectral triple with multitwisted real structure}
if there are $N$ densely defined operators $D_\ell$, $\ell=1,\ldots,N$, the  domains of which contain 
the domain of $D$, such that $\sum_{\ell=1}^N D_\ell =D$ 
and for every $\ell$ there exists an operator $\nu_\ell \in B(H)$, with bounded inverse, 
such that for every $a,b \!\in\!A$, 
the {\em multitwisted zero order condition} holds
\begin{equation} \label{mutw0oc}
[a, J \bar{\nu}_\ell(b) J^{-1} ] =  0 = [a, J \bar{\nu}_\ell^{-1}(b) J^{-1} ], 
\end{equation}
where $\bar{\nu}_\ell:=Ad_{\nu_\ell} \in Aut(B(H))$.
Additionally, if the spectral triple is even we assume 
\begin{equation} \label{mutwgamma}
\gamma \nu_\ell^2 = \nu_\ell^2 \gamma, \quad \forall\,\ell .
\end{equation}      
We say that  {\em multitwisted first-order condition} holds if 
\begin{equation} \label{mutw1oc}
[D_\ell, a] J \bar{\nu}_\ell(b) J^{-1} =  J \bar{\nu}_\ell^{-1}(b) J^{-1} [D_\ell, a],
\end{equation}
and that  {\em multitwisted $\epsilon'$ condition} holds if   
\begin{equation} \label{mutwec}
D_\ell J \nu_\ell = \epsilon' \nu_\ell J D_\ell, \quad
{\rm where} \quad \epsilon' = \pm 1
\end{equation}
and we call the multitwisted real spectral triple {\em regular} if
\begin{equation} \label{mutwregc}
\nu_\ell J \nu_\ell = J,
\end{equation}
for each $\ell$.
{\qquad$\diamond$}
\end{definition}
\begin{remark}
We assumed that the domain of the full Dirac operator $D$ is contained in the domains 
of all operators $D_\ell$ so the decomposition makes sense at least on this domain.
However, in principle it is not required that individually each $D_\ell$ is selfadjoint with 
compact resolvent and each $[D_\ell,a]$ is bounded as in order to obtain a spectral 
triple only the sum $D$ of all $D_\ell$ is required to have these properties. It will be interesting to find examples of this situation.
\quad$\diamond$
\end{remark}

\begin{remark} 
The notion of a spectral triple  with a twisted real structure in \cite{BCDS16,BDS19} fits this definition 
as a special case when $N\!=\!1$ and $\bar\nu_1$ is an automorphism of $A$, since then 
the multitwisted zero order condition \eqref{mutw0oc} is equivalent with \eqref{0oc}, and the relation 
(\ref{mutw1oc}), though it appears slightly different, is equivalent with the previous twisted order one condition by taking $b=\bar{\nu}(c)$.
Definition \ref{multi} is however slightly more general 
as we do not assume that $\bar{\nu_\ell}$ are automorphisms of the algebra $A$, and in order that the multitwisted 
first order condition \eqref{mutw1oc} be satisfied for all one forms, that is if $\omega_\ell = \sum_i a_i [D_\ell, b_i]$ 
then ${\omega_\ell J \bar{\nu}_\ell(b) J^{-1} =  J \bar{\nu}_\ell^{-1}(b) J^{-1} \omega_\ell}$,
we require that besides \eqref{0oc} also ${[a, J \bar{\nu}_\ell(b) J^{-1} ]=0}$ holds. Furthermore, the consistency with the $A$-bimodule structure 
of 1-forms requires also that 
$[a, J \bar{\nu}_\ell^{-1}(b) J^{-1} ]=0$, however, thanks to the consistency under the adjoint operation in $A$, 
if $\nu_\ell^*\!=\!\nu_\ell$, it suffices to impose one of these two conditions. 
\quad$\diamond$ 
\end{remark}

\subsection{Properties of spectral triples with  multitwisted real structure.}

The important feature of the spectral triples which are multitwisted real is that they are closed under the product.
\begin{proposition}\label{tensor}
Let $(A',H',D',J',\gamma')$  and $(A'',H'',D'',J'')$ be spectral triples with multitwisted real structure (the first one even and satisfying $J'\gamma'=\gamma'J'$), 
with $D' =\sum_{j=1}^{N'} D_j'$ and $D'' = \sum_{k=1}^{N''} D_k''$, 
for the twists $\nu'_j\!\in\!B(H')$ and $\nu_k''\!\in\!B(H'')$, respectively. 
Then 
\begin{equation}\label{prod}
{(A' \otimes A'', H' \otimes H'', D' \otimes \id + \gamma' \otimes D'', J' \otimes J'')}
\end{equation}
is a multitwisted real spectral triple with the Dirac operator decomposing as a sum of 
\begin{equation}
D_\ell = \begin{cases} D'_\ell \otimes \id, & 1 \leq \ell \leq N' \\
\gamma' \otimes {D}_{\ell-N'}'', & N'+1 \leq \ell \leq N'+N''
\end{cases}
\end{equation}
for the twists 
\begin{equation}
\nu_\ell = 
\begin{cases} \nu'_\ell \otimes \id, & 1 \leq \ell \leq N' \\
\id \otimes \nu''_{\ell-N'}, & N'+1 \leq \ell \leq N'+N''.
\end{cases}
\end{equation}
Furthermore, if both triples satisfy the multitwisted zero or first order conditions, and are regular then 
this holds for their tensor product.
\end{proposition}
\begin{proof}
It is well known that \eqref{prod} is a real spectral triple. 
The first equality in condition \eqref{mutw0oc} for $1\leq \ell \leq N'$
follows from \eqref{mutw0oc} for the first spectral triple and 
\eqref{0oc} for the second one
$$ 
(a'\otimes a'') (J'\bar{\nu}_\ell(b')J'^{-1}\otimes J''b'' J''^{-1}) =  (J'\bar{\nu}_\ell(b')J'^{-1}\otimes J''b'' J''^{-1}) (a'\otimes a'') $$
and analogously for $N'+1 \leq \ell \leq N'+N''$,
and similarly for the the second equality.
\\
The condition \eqref{mutw1oc} for $1\leq \ell \leq N'$ reads
$$ 
\begin{aligned}
([D_\ell, a']\otimes a'') &
(J'\bar{\nu}_\ell(b')J'^{-1}\otimes J''b'' J''^{-1}) = \\
&=  (J'\bar{\nu}_\ell(b')J'^{-1}\otimes J''b'' J''^{-1})
([D_\ell, a']\otimes a'')
\end{aligned}
$$
and is satisfied by \eqref{mutw1oc} for the first spectral triple and \eqref{0oc} for the second one,
and analogously for $N'+1 \leq \ell \leq N'+N''$ using the properties of $\gamma'$.
The condition \eqref{mutwec} for $1\leq \ell \leq N'$ (respectively $N'+1 \leq \ell \leq N'+N''$)
follows from \eqref{mutw0oc} for the first (respectively second) spectral triple.
Finally, \eqref{mutwregc} is immediate.
\end{proof}
\begin{remark}
Note that the resulting spectral triple is not even (as a product of an even and an odd triple). All 
other cases of the product of even and odd spectral triples can be also considered and we 
postpone the full discussion till future work.
\end{remark}
In \cite{BCDS16} we have demonstrated that, with an appropriate definition of
the fluctuated Dirac operator, a perturbation of $D$ by a one form and its
appropriate image in the commutant of the algebra $A$ yields the Dirac operator with the same properties. 
This functorial property holds also in the multitwisted case.
\begin{proposition}
Assume that $(A,H,D,J)$, where $D=\sum_{\ell=1}^{N} D_\ell$, is a spectral triple with multitwisted real structure satisfying the twisted zero and first-order conditions \eqref{mutw0oc}, \eqref{mutw1oc}. Let $\omega = \sum_i a_i [D,b_i]$ 
be a selfadjoint one-form. Then $(A,H,D_\omega,J)$, where $D_\omega=D\!+\!\omega$, 
is again a multitwisted real spectral triple satisfying the twisted  zero \eqref{mutw0oc} and first order condition \eqref{mutw1oc}, 
with $D_\omega = \sum_{\ell=1}^{N} (D_\omega)_\ell$, where
$(D_\omega)_\ell=D_\ell + \omega_\ell $ and
$\omega_\ell = \sum_i a_i [D_\ell, b_i] $, and with the same twists. Moreover, 
if $(A,H,D,J)$ is regular then so is $(A,H,D_\omega,J)$. 
\end{proposition}
\begin{proof}
Observe that neither \eqref{mutw0oc} nor the regularity condition \eqref{mutwregc} change, so we need 
to verify only the multitwisted first-order condition. Further, for any $\omega_\ell$ we have,
$$ 
\begin{aligned}
\omega_\ell  J \bar{\nu}_\ell(b) J^{-1} & = \sum_i a_i [D_\ell, b_i] J \bar{\nu}_\ell(b) J^{-1} = \sum_i a_i J \bar{\nu}^{-1}_\ell(b) J^{-1}  [D_\ell, b_i] \\
& =  \sum_i  J \bar{\nu}^{-1}_\ell(b) J^{-1}  a_i [D_\ell, b_i]  = J \bar{\nu}^{-1}_\ell(b) J^{-1} \omega_\ell.
\end{aligned}
$$
Since $[\omega_\ell, a] = \sum_i a_i [D_\ell, b_i a] - \sum_i (a_i b_i) [D_\ell, a] - a  \sum_i a_i [D_\ell, b_i ]$, 
for any $a\!\in\!A$, then it is of the same form as $\omega_\ell$, and in consequence we have,
$$
\begin{aligned}
\phantom{x} [(D_\omega)_\ell, a] J \bar{\nu}_\ell(b) J^{-1} & = [D_\ell + \omega_\ell, a] J \bar{\nu}_\ell(b) J^{-1} \\
& = [D_\ell, a] J \bar{\nu}_\ell(b) J^{-1} +  [\omega_\ell, a]  J \bar{\nu}_\ell(b) J^{-1}  \\
& = J \bar{\nu}^{-1}_\ell(b) J^{-1} [D_\ell, a]  +  J \bar{\nu}^{-1}_\ell(b)  [\omega_\ell, a]  \\
& = J \bar{\nu}^{-1}_\ell(b) [(D_\omega)_\ell, a].
\end{aligned}
$$
\end{proof}
\begin{remark}
Note that the above construction 
does not preserve the multitwisted $\epsilon'$-condition  \eqref{mutwec}.
To cure this problem, we modify the manner of fluctuations
of $D$.
\end{remark}
\begin{proposition}
Let $(A, H, D,J)$ be a spectral triple with multitwisted real structure satisfying \eqref{mutw0oc} and \eqref{mutw1oc}. Let $\omega$ be a one-form as in the proposition above. If the sum,
\begin{equation}
\sum_{\ell=1}^N \nu_\ell J (\omega_\ell) J^{-1} \nu_\ell, \label{asspt}
\end{equation}
is bounded, then taking:
\begin{equation}
(D_\omega)'_\ell =  D_\ell +  \omega_\ell 
+ \epsilon' \nu_\ell J (\omega_\ell) J^{-1} \nu_\ell , 
\end{equation}
and 
\begin{equation}
D_\omega' = \sum_{\ell=1}^N (D_\omega)'_\ell,
\end{equation}
$(A, H, D_\omega',J)$ is a spectral triple with a multitwisted real structure, 
with the same twists, satisfying the conditions \eqref{mutw0oc}, \eqref{mutw1oc} 
and \eqref{mutwec}.
Moreover, for each $a \!\in\!A,$
\begin{equation}
[(D_\omega)'_\ell ,a] = [(D_\omega)_\ell, a]. \label{samecomm}
\end{equation} 
\end{proposition}
\begin{proof}
Let us take $\omega_\ell = \sum_i a_i [D_\ell, b_i]$ Then for any $a\!\in\!A$,
$$
\begin{aligned}
\nu_\ell J (\omega_\ell) J^{-1} \nu_\ell a &= \nu_\ell J \left( \sum_i a_i [D_\ell, b_i] \right) J^{-1} \nu_\ell a \\
&=  \nu_\ell J \left(  \sum_i a_i [D_\ell, b_i] \right) (J^{-1} \bar{\nu}_\ell(a) J) J^{-1} \nu_\ell \\
& = \nu_\ell J \left( J^{-1} \bar{\nu}^{-1}_\ell(a) J\right) \left(  \sum_i a_i [D_\ell, b_i] \right) J^{-1} \nu_\ell \\
& = a \nu_\ell J \left(  \sum_i a_i [D_\ell, b_i] \right) J^{-1} \nu_\ell,
\end{aligned}
$$
where, again, we have used \eqref{mutw0oc} and \eqref{mutw1oc}. This shows, that for any $\omega_\ell$ and any $a$ \eqref{samecomm} holds and as a consequence
the bimodule of one-forms remains unchanged if we pass from $D_{\omega}$ to $D'_\omega$.

To see the multitwisted $\epsilon'$ condition \eqref{mutwec}, we check,
$$ 
\begin{aligned}
D_\ell' J \nu_\ell & = \left( D_\ell +  \omega_\ell 
+ \epsilon' \nu_\ell J (\omega_\ell) J^{-1} \nu_\ell \right)  J \nu_\ell \\
& = 
\epsilon' \nu_\ell J D_\ell + \omega_\ell  J \nu_\ell  +  \epsilon' \nu_\ell J (\omega_\ell) J^{-1} \nu_\ell J \nu_\ell  \\
& =\epsilon' \nu_\ell J D_\ell  + \omega_\ell  J \nu_\ell +  \epsilon' \nu_\ell J \omega_\ell  \\
& = \epsilon' \nu_\ell J \left(D_\ell + \epsilon' \nu_\ell J^{-1}  \omega_\ell  J \nu_\ell + \omega_\ell \right),
\end{aligned}
$$
where we have used \eqref{mutwregc} and $\epsilon'^2=1$ and $J^2=1$.
\end{proof}

\begin{remark}
Observe that the additional assumption in (\ref{asspt}) was necessary only 
to guarantee that the additional term is bounded. If this is not the case all properties 
of the modified Dirac operator that were demonstrated in the proof above will still hold, 
however, since the fluctuation is not by a bounded operator, spectral properties of 
$(D_\omega)'$ (like the compactness of the resolvent) may be modified.

It is also worth noting that one can extend the possible fluctuations of the Dirac operator 
to the sums of partial fluctuations that is, fluctuating each of $D_\ell$ by $\omega_\ell$, 
which may differ each from other, provided that the resulting full Dirac operator $D_\omega'$
is a bounded perturbation of $D$. 
\end{remark}
The notion of a spectral triple with a twisted real structure in \cite{BCDS16,BDS19} was largely
motivated by spectral triples conformally rescaled by a positive element in $JAJ^{-1}$.
Below we propose a generalization of this construction.
\begin{proposition}\label{multiconf}
Suppose that $(A,H,D,J)$ is a real spectral triple and 
${D= \sum_{\ell=1}^N D_\ell}$
such that each $D_\ell$ satisfy the first order condition for every ${\ell=1,\ldots N}$. 
Let $k_\ell$ be positive elements from $A$ with bounded inverses. 
Then $(A,H,\tilde{D},J)$, where 
\begin{equation}
\tilde{D} := (D_1)_{k_1}+ \cdots + (D_N)_{k_N},
\end{equation}
with $(D_\ell)_{k_\ell}= (Jk_\ell J^{-1}) D_\ell (Jk_\ell J^{-1})$
satisfies multitwisted zero \eqref{mutw0oc} and first order \eqref{mutw1oc} conditions with 
$$\nu_\ell= k_\ell^{-1}\, Jk_\ell J^{-1},$$
is regular \eqref{mutwregc} with the multitwisted $\epsilon'$-condition \eqref{mutwec}. Furthermore, 
if $D$ has compact resolvent it is a spectral triple with a multitwisted real structure.
\end{proposition}
\begin{proof}
Since the original real spectral triple satisfies the zero order condition \eqref{0oc}, so does the 
multitwisted-real spectral triple. Next, if all $k_\ell\!\in\!A$ then ${\nu_\ell(a) = k_\ell^{-1}a k_\ell \!\in\!A}$ for every $a\!\in\!A$ and 
\eqref{mutw0oc} holds as  well. We compute further using the first order condition for the spectral triple,
$$ 
\begin{aligned}
\ [(D_\ell)_{k_\ell}, a]  J \bar{\nu}_\ell(b) J^{-1} &=  (J k_\ell J^{-1}) [D_\ell, a] (Jk_\ell J^{-1}) J \bar{\nu}_\ell(b) J^{-1} \\
& = (Jk_\ell J^{-1})  [D_\ell, a]   (J k_\ell  J^{-1}) (J k_\ell^{-1} b  k_\ell J^{-1}) \\
& = (Jk_\ell J^{-1})  [D_\ell, a]  (J b J^{-1}) (J k_\ell J^{-1})\\
& = (Jk_\ell J^{-1})  (J b J^{-1})  [D_\ell, a]  J k_\ell J^{-1} \\
& = J \bar{\nu}^{-1}_\ell(b) J^{-1}   J k_\ell J^{-1}  [D_\ell, a]  J k_\ell J^{-1} \\
& = J \bar{\nu}^{-1}_\ell(b) J^{-1} [(D_\ell)_{k_\ell}, a].
\end{aligned}
$$
which proves \eqref{mutw1oc}. The regularity condition, \eqref{mutwregc} follows directly,
$$ \nu_\ell J \nu_\ell =  \left( k_\ell^{-1}\, Jk_\ell J^{-1} \right) J  \left( k_\ell^{-1}\, Jk_\ell J^{-1}\right)  
= k_\ell^{-1}\, Jk_\ell k_\ell^{-1}\, Jk_\ell J^{-1} = J,$$
using $J^2 = \epsilon$ so that $J = \epsilon J^{-1}$.

The multitwisted $\epsilon'$ condition is again a simple consequence of the $\epsilon'$-condition of the
real spectral triple:
$$ 
\begin{aligned}
(D_\ell)_{k_\ell} J \nu_\ell &= (D_\ell)_{k_\ell}  J (k_\ell^{-1}\, Jk_\ell J^{-1}) \\
& = (Jk_\ell J^{-1}) D_\ell (Jk_\ell J^{-1}) J (k_\ell^{-1}\, Jk_\ell J^{-1}) \\
&=  J k_\ell J^{-1} D_\ell  k_\ell J = \epsilon' J k_\ell  D_\ell  (J k_\ell J^{-1}) \\
&=  \epsilon' J k_\ell  (J k_\ell^{-1} J^{-1}) (D_\ell)_{k_\ell} \\
& =  \epsilon' (k_\ell^{-1} J k_\ell J^{-1}) J (D_\ell)_{k_\ell} = 
 \epsilon' \nu_\ell J (D_\ell)_{k_\ell}.
\end{aligned}
$$
\end{proof}

\begin{remark}
It is worth noticing that the above construction mimics the conformal rescaling of  the Dirac operator from
\cite{BCDS16}, however, additionally one still has to assume that the resulting Dirac operator has compact
resolvent. It is an interesting problem if and under what conditions on the operators $D_\ell$ this occurs for 
a spectral triple that allows such splitting of $D$.
\end{remark}

\section{Examples}
\subsection{Multiconformally rescaled spectral triples.} \label{crst}
A specific example of the above construction  was given by the asymmetric torus in \cite{DaSi15}. 
It was motivated by the search of spectral triples over the noncommutative torus which can be interpreted 
as arising from a non-flat metric. In fact, in a certain precise sense it has a non-vanishing local scalar curvature, 
yet obeying a generalized Gauss-Bonnet theorem.  We will supplement the construction in \cite{DaSi15} by discussion 
of the real structure and twisted reality properties.

Let $\partial_1$ and $\partial_2$ denote the operators that extend the standard derivations of 
$C^\infty(T^2_\theta)$ to $H= L^2(T^2_\theta)\otimes \mathbb{C}^2$ as selfadjoint (unbounded) 
operators and $J$ be the usual antilinear isometry on $H$. Then for any positive 
invertible $k_1,k_2 \in C^\infty(T^2_\theta)$, the Dirac operator
\begin{equation}
\tilde{D} = Jk_1 J^{-1}\sigma^1 \partial_1 Jk_1 J^{-1} + 
Jk_2 J^{-1}\sigma^2 \partial_2 Jk_2 J^{-1},
\end{equation}
where $\sigma^1$ and $\sigma^2$ are the usual Pauli matrices, makes 
$(C^\infty(T^2_\theta), L^2(T^2_\theta)\otimes \mathbb{C}^2, \tilde{D}, J)$ a multitwisted 
real spectral triple  satisfying by Proposition\,\ref{multiconf} all conditions including
\eqref{mutw0oc}, \eqref{mutw1oc}, \eqref{mutwregc}.
In \cite{DaSi15} we considered a particular case with $k_1\!=\!1$
(which is not a product spectral triple). A four-dimensional generalization (of product type) with two 
different scalings was studied in \cite{CoFa13}.
\subsection{Conformal rescaling without an automorphism}
Consider the following situation, which further generalizes the construction of conformally
rescaled spectral triples
\footnote{Note that here conformally rescaled spectral triples are indeed {\em spectral triples} 
and not {\em twisted spectral triples}. We still call them conformally rescaled as they are such in the classical (commutative) situation.} 	
allowing conformal rescaling of the Dirac operator by an element,
which is still from the commutant of $A$ but not from $JAJ^{-1}$. We have:
\begin{proposition}	
Let $(A,H,D,J)$ be a real spectral triple, which satisfies the usual first order condition, regularity
and $\epsilon'$ condition. Let $Cl_D(A)$ be the algebra generated by $A$ and $[D,A]$ and $k\!\in\! Cl_D(A)$ be 
an invertible  element with bounded inverse.  
Then with $\nu = k^{-1} JkJ^{-1}$ and $D_k = JkJ^{-1} D JkJ^{-1}$, $(A,D_k,H,J,\nu)$
with a twist $\nu$ is a spectral triple with a twisted (and thus a multitwisted, in the sense of 
Def. \ref{multi}) real structure, which satisfies the twisted zero and first order conditions, twisted 
regularity and twisted $\epsilon'$-condition.
\end{proposition}
\begin{proof}
First of all, observe that since 	$k\!\in\! Cl_D(A)$ and the spectral triple satisfies the first order condition
then $\bar{\nu}(b) = k^{-1} b k$ for any $b\!\in\!A$, however,  $\bar{\nu}$ is not necessarily an automorphism 
of $A$. Still, a simple computation using the first order condition and the fact that $k \!\in\! Cl_D(A)$ shows 
that
\begin{equation}
\begin{aligned}
a J \bar{\nu}(b) J^{-1} &=  a J k^{-1} b k J^{-1} = J (J^{-1}a J) (k^{-1} b k ) J^{-1} \\
& = J  (k^{-1} b k )  J^{-1} a = J \bar{\nu}^{-1}(b)J^{-1} a,
\end{aligned}
\end{equation}
which is \eqref{mutw0oc}. Further, using again zero and first order conditions for the spectral triple we have,
\begin{equation}
[D_k,a] J \bar{\nu}(b) J^{-1} = J \bar{\nu}^{-1}(b)J^{-1} [D_k,a],
\end{equation}
for any $a,b \!\in\!A$, which is precisely \eqref{mutw1oc}. Unlike in the case of the usual conformal rescaling one cannot
write this condition replacing $b$ with $c=\bar{\nu}(b)$ as $c$ is not guaranteed to be in $A$. The proof of regularity and
the twisted $\epsilon'$ condition are the same as in the standard situation of conformally rescaled spectral triple. 
\end{proof}

The above construction is, of course, a case of single twisting, however, can be easily extended to the
situation of multitwisting and multiconformal scaling, which provides new examples of multitwisted real 
spectral triples. Interestingly, such objects do have a deep geometric motivation, arising from the Dirac 
operators over noncommutative circle bundles \cite{DaSi03}.
 
An example is a three-dimensional noncommutative torus $\TT^3_\theta$ seen as the 
noncommutative $U(1)$-bundle over the two-dimensional  noncommutative torus $\TT^2_\theta$. 
We consider the usual equivariant Dirac
operator $D$ over $\TT^3_\theta$ and the bimodule of one-forms in
the Clifford algebra $Cl_D(\TT^3_\theta)$. There exists a  canonical action of $U(1)$ on $\TT^3_\theta$, as described 
in \cite{DaSi03}, the invariant subalgebra of which is $\TT^2_\theta$.
A $U(1)$-connection over $C^\infty(\TT^3_\theta)$ can be given by a one-form, 
which as an element of $Cl_D(\TT^3_\theta)$ is
\begin{equation}
\omega = \sigma^1 \omega_1 + \sigma^2 \omega_2 + \sigma^3, 
\label{astco}
\end{equation}
where $\omega_1,\omega_2 \in \TT^2_\theta$ are $U(1)$-invariant elements of the algebra $C^\infty(\TT^3_\theta)$.

In  \cite{DaSi03} we have shown that for any selfadjoint connection $\omega$ (\ref{astco}) there 
exists a compatible (in the sense defined therein) Dirac operator over $A=C^\infty(\TT^3_\theta)$, which has 
the form 
\begin{equation}
{\mathcal D}_{\omega} = \sigma^1 \partial_1 + \sigma^2 \partial_2  + J w J^{-1} \partial_3,  
\end{equation}
where $J$ is the usual real structure on $C^\infty(\TT^3_\theta)$, $\partial_i$, $i=1,2,3$, are the usual
derivations represented on the Hilbert space of the spectral triple and the one-form $w$ is,
$$ w = \sigma^3 - \sigma^1 \omega_1 - \sigma^2 \omega_2.$$ 
Although ${\mathcal D}_\omega$ does not satisfy a twisted first order condition, we have:

\begin{proposition}
The spectral triple $(C^\infty(\TT^3_\theta), H, {\mathcal D}_{\omega}, J)$ is a multitwisted 
real spectral triple, with splitting ${\mathcal D}_{\omega} = D_{(2)} + D_w$ and twists, 
$$ \nu_1=\id, \qquad \nu_2  = w^{-\frac{1}{2}} J w^{\frac{1}{2}} J^{-1}.$$
\end{proposition}
\begin{proof}
The decomposition of $D$ is natural, with 	
$$D_{(2)}=  \sigma^1 \partial_1 + \sigma^2 \partial_2, $$
being the usual Dirac operator over the noncommutative two-torus. Clearly, it satisfies the first order
condition and the $\epsilon'$ condition for the trivial twist $\nu_1=\id$. The second part of the splitting
$$ D_w=J w J^{-1} \partial_3 ,$$ 
can be rewritten as
$$ D_w=J w^{\frac{1}{2}} J^{-1} \partial_3 J w^{\frac{1}{2}} J^{-1},$$ 
since $w$ is invariant with respect to the $U(1)$-action and therefore commutes with $\partial_3$. Note 
that since $w$ is hermitian it has a square root and we can use it to write $D_w$ in a convenient form.
It is easy to see that the decomposition and the twists satisfy 
\begin{equation}
\begin{aligned}
&[D_{(2)},a] J b J^{-1} = J  bJ^{-1} [D_{(2)},a], \\
&[D_w,a] J \bar{\nu} (b) J^{-1} = J \bar{\nu}^{-1}( b)J^{-1} [D_w,a],
\end{aligned}
\end{equation}
where 
$\bar{\nu}(x) = w^{-\frac{1}{2}} x w^{\frac{1}{2}}$ and hence the requirements of the definition \ref{multi}.
Note that the condition \eqref{mutw0oc} is also satisfied since $w$ belongs to the completion of Clifford algebra 
and therefore $J\bar{\nu}(a)J^{-1}$ is in the commutant of $A$.  
\end{proof}
Note that since by construction $w\!\notin\!C^\infty(\TT^3_\theta)$ this multitwisted 
spectral triple is not of the same type as the example discussed in the subsection \ref{crst}.

\section{Conclusions and outlook}

The new notion of spectral triples with a multitwisted real structure, which we propose here, has some major advantages. 
Firstly, it is consistent with the usual definition of {\em spectral triples} (unbounded Fredholm modules)
thus, allowing to use the power of Connes-Moscovici local index theorem. Secondly it vastly extends the realm of examples, 
covering almost all known spectral triples, including those motivated by geometrical constructions, like conformal rescaling 
or  noncommutative principle fibre bundles. Moreover, it is closed under the tensor product operation.

In particular, the multitwisted first order condition may provide a better understanding of the notion of first order differential 
operators in noncommutative geometry, which we hope will allow to finer apprehend the examples arising from the 
quantum groups  and quantum homogeneous spaces as constituting noncommutative manifolds. 

\end{document}